\documentclass[11pt]{amsart}
\usepackage[margin=1in]{geometry}  

\usepackage{amsmath, amsthm, latexsym, amssymb, amsfonts}
\usepackage[pdftex]{graphicx,color} 
\usepackage{algorithm}
\usepackage{algpseudocode}
\usepackage[all]{xy}

 \theoremstyle{plain}
 \newtheorem{tm}{Theorem}[section]
 \newtheorem{lm}[tm]{Lemma}
\newtheorem{coro}[tm]{Corollary}
 \newtheorem{pro}[tm]{Proposition}

 \theoremstyle{definition}
 \newtheorem{defi}[tm]{Definition}
 \newtheorem{rema}[tm]{Remark}
 
  \newtheorem{ex}[tm]{Example}

 \newtheorem*{th*}{Theorem}

\newcommand{\cl}[1]{\mathcal{#1}}

\newcommand{\Z}{\mathbb Z}
\newcommand{\N}{\mathbb N}
\newcommand{\R}{\mathbb R}
\newcommand{\F}{\mathbb F}

\newcommand{\cA}{\cl A}

\newcommand{\cC}{\cl C}

\newcommand{\K}{{\mathbb K}}

\newcommand{\la}{\langle}
\newcommand{\ra}{\rangle}

\newcommand{\sig}{\sigma}
\newcommand{\Sig}{\Sigma}

\newcommand{\m}{\mathfra\K{m}}

\newcommand{\Ker}{\operatorname{Ker}}
\newcommand{\im}{\operatorname{Im}}

\newcommand{\Cl}{\operatorname{Cl}}

\newcommand{\dis}{\displaystyle}
\def\aa{{\bf \alpha}}
\def\bb{\beta}
\def\a{{\bf a}}
\newcommand{\be}{\mathbf b}
\def\t{{\bf t}}

\def\uu{{\bf u}}
\def\vv{{\bf v}}
\def\x{{\bf x}}

\def\m{{\bf m}}

\def\ev{{\text{ev}}}

\newcommand{\cK}{ \K^r}
  \newcommand{\height}{\operatorname{height}}
    \newcommand{\rank}{\operatorname{rank}}

 \usepackage{graphics}
 \usepackage[english]{babel}

\usepackage[latin1]{inputenc}

\usepackage{times}
\usepackage[T1]{fontenc}
 \usepackage{graphics}

\begin{document}


\title{Toric Codes and Lattice Ideals}
 \thanks{The author is supported by T\"{U}B\.{I}TAK Project No:114F094}

\author[Mesut Sahin]{Mesut \c{S}AH\.{I}N}
\address[Mesut Sahin]{Department of Mathematics, Hacettepe  University, Ankara,TURKEY}
\email{mesut.sahin@hacettepe.edu.tr}
\keywords{Evaluation code, toric code, lattice ideal, toric ideal, Hilbert function}
\subjclass[2010]{Primary 14M25, 14G50; Secondary 52B20}


\begin{abstract} Let $X$ be a complete simplicial toric variety over a finite field $\F_q$ with homogeneous coordinate ring $S=\F_q[x_1,\dots,x_r]$ and split torus $T_X\cong (\F^*_q)^n$. We prove that vanishing ideal of a subset $Y$ of the torus $T_X$ is a lattice ideal if and only if $Y$ is a subgroup. We show that these subgroups are exactly those subsets that are parameterized by Laurents monomials. We give an algorithm for determining this parametrization if the subgroup is the zero locus of a lattice ideal in the torus. We also show that vanishing ideals of subgroups of $T_X$ are radical homogeneous lattice ideals of dimension $r-n$. We identify the lattice corresponding to a degenerate torus in $X$ and completely characterize when its lattice ideal is a complete intersection. We compute dimension and length of some generalized toric codes defined on these degenerate tori.
\end{abstract}

\maketitle

\section{Introduction}
Let $X$ be a complete simplicial toric variety over a field $\K$, corresponding to a fan $\Sig\subset \R^n$ and let $T_X \cong (\K^*)^n$  be split which is the case for instance when $\K$ is algebraically closed or finite. Denote by $\rho_1,\dots,\rho_r$ the rays in $\Sig$ and $\vv_1,\dots,\vv_r\in \Z^n$ the corresponding primitive lattice vectors generating them. Given a vector $\uu=(u_1,\dots,u_r)$ we use $\mathbf{t}^\uu$ to denote the monomial 
 $\mathbf{t}^\uu=t_1^{u_1}\dots t_r^{u_r}$. Recall the following dual exact sequences:
$$\dis \xymatrix{ \mathfrak{P}: 0  \ar[r] & \Z^n \ar[r]^{\phi} & \Z^r \ar[r]^{{\bb}} & \cA \ar[r]& 0},$$  
where $\phi$ is the matrix with rows $\vv_1,\dots,\vv_r$, and
$$\dis \xymatrix{ \mathfrak{P}^*: 1  \ar[r] & G \ar[r]^{i} & (\K^*)^r \ar[r]^{\pi} & T_X \ar[r]& 1},$$
where $\pi:(t_1,\dots, t_r)\mapsto [\mathbf{t}^{\uu_1}: \cdots : \mathbf{t}^{\uu_n}],$ with $\uu_1,\dots, \uu_n$ being the columns of $\phi$ and $G=\Ker(\pi)$.

Let
$\dis S=\K[x_1,\dots, x_r]=\bigoplus_{\aa \in \cA} S_{\aa}$ 
be the homogeneous coordinate or Cox ring of $X$, multigraded by $\cA \cong \Cl(X)$ via $\bb_j:=\deg_{\cA}(x_j):=\bb(e_j)$, where $e_j$ is the standart basis element of $\Z^r$ for each $j=1,\dots,r$. 
The irrelevant ideal is a monomial ideal defined by $B=\la x^{\hat{\sig}} ~:~ \sig \in \Sig\ra$, where $\dis x^{\hat{\sig}}=\Pi_{\rho_i \notin \sig}^{}x_i$. Thus, $T_X \cong (\K^*)^r /G$ and $X\cong (\cK\setminus V(B)) /G$ as a geometric quotient if we additionally assume that the order of the torsion part of the group $\cA$ is coprime to the characteristic of the field $\K$, since $G$ will be reductive in this case, see \cite{Coxhom,CLSch}. This assumption is satisfied in particular when $\cA$ is free or $\K$ has characteristic zero. The homogeneous polynomials of $S$ are supported in the semigroup $\N\bb$ generated by $\bb_1,\dots,\bb_r$, i.e. $\dim_{\K}{S}_{\aa}=0$ when $\aa \notin \N\bb$. 

Next, we recall evaluation codes defined on subsets $Y=\{[P_1],\dots,[P_N]\}$ of the torus $T_X$. Fix a degree $\aa\in\N\beta$ and a monomial $F_0=\x^{\phi(\m_0)+\a} \in S_{\aa}$,
where  $\m_0\in \Z^n$, $\a$ is any element of $\Z^r$ with $\deg(\a)=\aa$, and $\phi$ as in the exact sequence $\mathfrak{P}$.
This defines the {\it evaluation map}
\begin{equation}\label{e:evalmap}
\ev_{Y}:S_\aa\to \F_q^N,\quad F\mapsto \left(\frac{F(P_1)}{F_0(P_1)},\dots,\frac{F(P_N)}{F_0(P_N)}\right).
\end{equation}
The image $\cC_{\aa,Y}=\text{ev}_{Y}(S_\aa)$ is a linear code, called the {\it generalized toric code}. The block-length $N$, the dimension $k=\dim_{\F_q}(\cC_{\aa,Y})$, and the minimum distance $d=d(\cC)$ are three basic parameters of $\cC_{\aa,Y}$. Minimum distance is the minimum of the number of nonzero components of nonzero vectors in $\cC_{\aa,Y}$. Toric codes was introduced for the first time by Hansen in \cite{Ha0, Ha1} and studied later in \cite{Jo,Ru} for the special case of $Y=T_X$. Clearly, the block-length of $\cC_{\aa,Y}$ equals $N=|T_X|=(q-1)^n$ in this case. But it is not known in the general case. An algebraic way to compute the dimension and length of a generalized toric code is given in \cite{sasop}. This method is based on the observation that the kernel of the evaluation map above is determined by the vanishing ideal of $Y$ defined as follows. For $Y\subset X$, we define the vanishing ideal $I(Y)$ of $Y$ to be the ideal generated by homogeneous polynomials vanishing on $Y$. $I(Y)$ is a {\it complete intersection} if it is generated by
a regular sequence of homogeneous polynomials $F_1,\dots, F_k\in S$ 
where $k$ is the height of $I_Y$. When the vanishing ideal $I_Y$ is a complete intersection, bounds on the minimum distance of $\cC_{\aa,Y}$  is provided in \cite{sop}. There are interesting results about evaluation codes on complete intersections in literature, e.g. \cite{DRT,GLSch,Ha3}. These results motivate studying vanishing ideals of special subsets of the torus $T_X$ and characterize when they are complete intersections. 

When the vanishing ideal is a binomial or a lattice ideal it is relatively easier to characterize whether it is a complete intersection, see \cite{MT} and references therein. In the case that the toric variety $X$ is a projective space over a finite field, there are interesting results relating lattice ideals and subtori of the torus $T_X$. Vanishing ideals of subgroups of $T_X$ parameterised by Laurent monomials are studied in \cite{RSV} and shown to be homogeneous lattice ideals of dimension $1$.  Binomial ideals that are vanishing ideals are characterised in \cite{TV}. It was also proven that homogeneous lattice ideals of dimension $1$ correspond to subgroups of the torus $T_X$. In \cite{NPV}, these subgroups are identified as being the subgroups of $T_X$ parameterized by Laurent monomials. In the case that the toric variety $X$ is weighted projective space over a finite field, vanishing ideal of the torus itself is shown to be a homogeneous lattice ideal of dimension $1$ in \cite{DN}. 

We use some of the main ideas in these works and extend the main results to more general toric varieties over any field in the present paper. In section $2$, we show that a homogeneous binomial ideal gives a submonoid of $T_X$, see Proposition \ref{p:monoid} and Corollary \ref{c:latticemonoid}. Conversely, we show that vanishing ideals of submonoids of $T_X$ are lattice ideals in Proposition \ref{p:binom}. We prove that vanishing ideal of a subset $Y$ of the torus $T_X$ is a lattice ideal if and only if $Y$ is a subgroup, for $\K=\F_q$, in Theorem \ref{t:subgpsLattices}. In section $3$, we focus on finite submonoids of the torus when $\K$ is algebraically closed or finite. We show that these submonoids coincide with the submonoids parameterized by Laurent monomials on a finite subgroup of $\K^*$, see Theorem \ref{t:parametrizingsubmonoids}. We also give an algorithm for parameterizing the submonoid obtained from a homogeneous lattice ideal, when $\K=\F_q$, see Proposition \ref{p:parameterisingV_L} and Algorithm $1$. In section $4$, we determine the homogeneous lattice defining the vanishing ideals of subtori of $T_X$ called the degenerate tori which are parameterized by diagonal matrices, see Theorem \ref{t:I(Y_A)}. We also characterize when these are complete intersection in Proposition \ref{p:ci}. In section $5$, we list main properties of vanishing ideals of arbitrary subsets of the torus $T_X$ and prove that they must be radical homogeneous ideals of dimension $r-n$, see Theorem \ref{t:primdec}. Finally, we give quick applications of these on evaluation codes on degenerate tori, see Corollaries \ref{c:dimcode} and \ref{c:lengthcode}.

\section{Lattice Ideals and Submonoids of the Torus} \label{S:latticeideals} 
In this section we explore the connection between homogeneous lattice ideals and submonoids of the torus $T_X$. 

Recall that $\x^{\uu}:=x_1^{u_1}\dots x_m^{u_m}$ denotes a Laurent monomial for a vector $\uu\in\Z^r$. A binomial is a difference 
$\x^{\a}-\x^{\be}$ of two monomials corresponding to the vectors $\a,\be\in \N^r$. An ideal is called a \textit{binomial ideal} if it is generated by binomials. An exact sequence $\mathfrak{P}$ gives a grading on the polynomial ring $S=\K[x_1,\dots,x_r]$, where $\deg_{\cA}(x_j):=\bb_j$ is the $j$-th column of the degree matrix $\bb$. A homogeneous ideal in $S$ with respect to this grading is called $\cA$-graded or $\cA$-homogeneous. We usually just say that it is \textit{homogeneous} if the group $\cA$ is clear from the context.

Throughout $[P]=G \cdot P$ denotes a point in $X$. We use $[1]$ to denote $[1:\cdots:1]$. Similarly, $[V(B)]$ denotes the set of elements $[P]$ suct that $P\in V(B)\subset \K^r$. If $[P],[P']\in X$ then $[P]\cdot[P']:=[PP']$ is a well-defined element of $X\cup [V(B)]$. The set $[V(B)]$ does not include $[1]$ but it is closed with respect to this coordinatewise multiplication operation as $x^{\hat{\sig}} (PP')=x^{\hat{\sig}} (P)x^{\hat{\sig}} (P')=0$ when $[P]\in [V(B)]$ or $[P']\in [V(B)]$. Although $X$ might not be closed with respect to this multiplication, $X\cup [V(B)]$ is a monoid with identity $[1]$. Finally, if $Y\subseteq X$ is a submonoid of $X\cup [V(B)]$, then so is $Y\cup [V(B)]$.

For a homogeneous ideal $J$ of $S$, let $$V_X(J):=\{[P]\in X : F(P)=0, \mbox{for all homogeneous}\: F\in J\}.$$

\begin{pro}\label{p:monoid} If $J\subset S$ is a homogeneous binomial ideal then $V_X(J)\cup [V(B)]$ is a submonoid of $X\cup [V(B)]$. If furthermore, $PP'\notin V(B)$, for all $[P],[P']\in V_X(J)$, then $V_X(J)$  is a submonoid of $X\cup [V(B)]$. In particular, $V_X(J)\cap T_X$ is a submonoid of $T_X$.
\end{pro}
\begin{proof} If $0\neq f=\x^{\a}- \x^{\be}$ is a generator of $J$, then $f([1])=1-1=0$. So, $[1]\in V_X(J)$. If $[P],[P']\in V_X(J)$, then $f(P)=f(P')=0$. This implies that $\x^\a(P)=\x^\be(P)$ and $\x^\a(P')=\x^\be(P')$. So, $\x^{\a}(PP')=\x^\be(PP')$ yielding $f(PP')=0$, i.e. $PP'\in V(J)$. Hence $[V(J)]=V_X(J)\cup [V(B)]$ is a submonoid. If furthermore, $PP'\notin V(B)$, for all $[P],[P']\in V_X(J)$, then $[PP']\in V_X(J)$ and thus $V_X(J)$ is a submonoid. When $[P],[P']\in V_X(J)\cap T_X$, we have $PP'\notin V(B)$ and so $[PP']\in V_X(J)\cap T_X$. Therefore, $V_X(J)\cap T_X$ is a submonoid.
\end{proof}

\begin{defi} A lattice $L$ is a subgroup of $\Z^r$ and is called {\bf homogeneous} if $L \subseteq L_{\bb}=\Ker \bb=\im \phi$. Recall that every vector in $\Z^r$ is written as $\m=\m^+ -\m^-$, where $\m^+ ,\m^-\in \N^r$. Letting $F_\m=\x^{\m^+}-\x^{\m^-} $, the \textbf{lattice ideal} $I_L$ is the binomial ideal $I_L=\langle F_\m \: | \: \m \in L \rangle$ generated by special binomials $F_\m$ arising from $L\subset \Z^r$. 
\end{defi}
 
 \begin{pro}\label{p:homogeneouslattice} $L$ is homogeneous if and only if $I_L$ is homogeneous. 
\end{pro}
\begin{proof} Since $I_L$ is generated by the set $\{F_\m \: | \: \m \in L\}$, it is homogeneous if and only if $F_\m$ is homogeneous for all $\m \in L$. But $F_\m=\x^{\m^+}-\x^{\m^-} $ is homogeneous $\iff \bb(\m^+)=\bb(\m^+)$ 
$\iff \bb(\m)=0$ $\iff \m \in L_{\bb}$.  Thus, $I_L$ is homogeneous if and only if $\m \in L_{\bb}$ for all $\m \in L$, that is, $L$ is homogeneous.
\end{proof}

\begin{coro}\label{c:latticemonoid} If $L$ is homogeneous, $V_X(I_L)\cap T_X$ is a submonoid of $T_X$.
\end{coro}
\begin{proof} This follows from Propositions \ref{p:monoid} and \ref{p:homogeneouslattice}.
\end{proof}

\begin{defi}\label{D:VanIdeal} 
For $Y\subseteq X$, we define the vanishing ideal $I(Y)$ of $Y$ to be the ideal generated by homogeneous polynomials vanishing on $Y$. 
\end{defi}

\begin{pro}\label{p:binom} Let $Y \subseteq X$. If $Y$ is a submonoid of $X\cup [V(B)]$, then $I(Y)$ is a binomial ideal. If $Y \subseteq T_X$ is a submonoid, $I(Y)$ is a lattice ideal. 
\end{pro}
\begin{proof} If $Y$ is a submonoid of $X\cup [V(B)]$, the set $\mathcal{Y}=\{P\in \K^r \:|\: [P]\in Y\}$ is a submonoid of $\K^r$. Take a homogeneous polynomial $f=\lambda_1 \x^{\a_1}+\cdots+\lambda_k \x^{\a_k}$ from $ I(Y)$ with exactly $k$ terms. As $[1]\in Y$, $f$ can not be a monomial, i.e. $k>1$. Since monomials are characters from $\mathcal{Y}$ to the multiplicative monoid $(\K,.,1)$, it follows from Dedekind's Theorem that the monomials of $f$ can not be distinct characters as in the proof of \cite[Theorem 3.2]{TV}. If, say, $\x^{\a_1}$ and $\x^{\a_2}$ are the same as characters, then $\x^{\a_1}-\x^{\a_2}\in I(Y)$. Thus $f-\lambda_1(\x^{\a_1}-\x^{\a_2})$ lies in $I(Y)$ and has exactly $k-1$ terms. By induction on $k$, $f$ must be written as an algebraic combination of binomials, and hence $I(Y)$ is a binomial ideal. For the second part, take a polynomial $f$ of $S$ such that $x_jf\in I(Y)$ for any $j=1,\dots,r$. By \cite[Corollary 2.5]{binomial}, it is enough to show that $f\in I(Y)$. Let $f=\bigoplus\limits_{\alpha\in {\N\beta}}f_\alpha$ be the expansion of $f$ as the sum of its homogeneous components where $f_\alpha$ has total degree $\alpha$. Then $x_j{f}=\bigoplus\limits _{\alpha\in {\N\beta}}x_j{f_\alpha}$ and $x_j{f_\alpha}$ is homogeneous of total degree $\alpha+ \beta_j$. As $I(Y)$ is a homogeneous ideal, $x_j{f_\alpha}\in I(Y)$.  So, $(x_jf_\alpha)(P)=p_jf_{\alpha}(P)=0$ for any element $[P] \in Y$. Since $P\in(\K^*)^r$, $p_j\neq 0$, and so $f_{\alpha}(P)=0$. This yields that $f_\alpha\in I(Y)$ for all $\alpha$, consequently $f\in I(Y)$ completing the proof.
\end{proof}

\begin{rema} It is not clear if the hypothesis of Proposition \ref{p:binom} can be replaced with the weaker hypothesis that $Y\cup [V(B)]$ is a submonoid of $X \cup [V(B)]$. When $X$ is a projective space or weighted projective space $V(B)=\{0\}$, so this can be done using the fact that if $f\in I(Y)$ is a homogeneous polynomial then it also vanishes on $V(B)$. See \cite[Theorem 3.2]{TV} for the case of projective space. 
\end{rema}

\begin{lm} \label{L:Zariski}We have the following basic properties.
\begin{enumerate}
\item $\emptyset=V_X(1)$ and $X=V_X(0)$.\\
\item $V_X(J_1)\cup V_X(J_2)=V_X(J_1J_2)$ and $\displaystyle \bigcap_{i} V_X(J_i)=V_X(\bigcup_i J_i)$.\\
\item $Y_1\subseteq Y_2 \implies I(Y_2)\subseteq I(Y_1)$ .\\
\item $J_1\subseteq J_2 \implies V_X(J_2)\subseteq V_X(J_1)$. \\
\item If $\bar{Y}$ is the Zariski closure of $Y$, then $ I(Y)= I(\bar{Y})$ .\\
\item $\bar{Y}= V_X(I(Y))$. \\
\end{enumerate}
\end{lm}
\begin{proof} These can be proven using the fact that homogeneous ideals are generated by homogeneous polynomials. For instance, if $[P]\in V_X(J_1J_2)\setminus V_X(J_1)$, then it suffices to show that $G(P)=0$ for every homogeneous $G\in J_2$. Since, there is some homogeneous $F\in J_1$ with $F(P)\neq 0$ and $F(P)G(P)=0$, this follows at once. So, $V_X(J_1J_2) \subseteq V_X(J_1)\cup V_X(J_2)$. The rest is either similar or obvious.
\end{proof}

We define Zariski topology on $X$ by taking closed sets to be the zero loci $V_X(J)$ for homogeneous ideals $J$. 

We close this section with establishing the following correspondence.

\begin{tm} \label{t:subgpsLattices}Let $Y \subseteq T_X$ and $\K=\F_q$. Then, $I(Y)$ is a lattice ideal if and only if $Y$ is a subgroup.
\end{tm}
\begin{proof} If $I(Y)$ is a lattice ideal then $V_X(I(Y))\cap T_X$ is a submonoid of $T_X$ by Proposition \ref{p:monoid}. Since $V_X(I(Y))=\bar{Y}=Y$ lies in $T_X$ and $\K$ is finite $Y$ is a subgroup. The converse follows directly from Proposition \ref{p:binom}.
\end{proof}

\section{Finite submonoids of the torus}

In this section we completely characterize finite submonoids of the torus $T_X$ when $\K$ is a finite field or an algebraically closed field by showing that they are exactly those that are parameterized by Laurent monomials. This was done when $X$ is a projective space over a finite field by \cite{NPV}  and over an algebraically closed field by \cite{TV}. We give an algorithm producing these monomials when $\K$ is a finite field and the submonoid corresponds to a homogeneous lattice ideal.

\begin{defi} Let $H$ be a subset of $\K^*$ and $Q$ be an $s\times r$ integer matrix with entries $q_{ij}$. Then the parameterized set corresponding to $Q$ and $H$ is defined as follows:
 $$Y_{Q,H}=\{[t_1^{q_{11}}\cdots t_s^{q_{s1}}:\cdots :t_1^{q_{1r}}\cdots t_s^{q_{sr}}]: t_1,\dots,t_s \in H\}.$$
\end{defi}

\begin{tm}\label{t:parametrizingsubmonoids} Let $\K$ be a finite field or an algebraically closed field. Then the following are equivalent:

\begin{enumerate}
\item $Y$ is a \textbf{finite} submonoid of $T_X$.\\
\item $Y =Y_{Q,H}$ for a \textbf{finite} submonoid $H$ of $\K^*$.
\end{enumerate}
\end{tm}

\begin{proof} Suppose that $(1)$ holds. In this case $Y$ is a finite subgroup as being a finite sumbonoid of a group. Thus $Y$ is a product $\la [P_1] \ra \times\cdots \times \langle [P_s] \rangle$ of cyclic subgroups generated by points $[P_1],\dots, [P_s]\in Y$ of orders, say,  $c_1,\dots,c_s$. In this case, for every $i=1,\dots,s$, we have $[P_i]^{c_i}=[1]=G$. If $\K$ is finite, we define $H$ to be the subgroup generated by the coordinates $p_{ij}$ of $P_1,\dots, P_s$. Then $H$ is a finite subgroup of $\K^*$. If $\K$ is algebraically closed, then we define $H$ to be the subgroup generated by all $x\in \K$ such that $x^{c_i}=1$, for some $i=1,\dots,s$. We show that $H$ is non-empty. Since $(1,\dots,1)\in G=[P_i^{c_i}]$ and $G$ is parameterized by the columns of $\bb$, it follows that $(1,\dots,1)=(\mathbf{t}^{\bb_1}p_{i1}^{c_i},\dots,\mathbf{t}^{\bb_r}p_{ir}^{c_i})$. So, $\mathbf{t}^{\bb_j}p_{ij}^{c_i}=1$, for all $j=1,\dots,r$. Let ${\mu}_{jk}\in \K^*$ be such that ${\mu}_{jk}^{c_i}=t_k$, for $j=1,\dots,r$ and $k=1,\dots,r-n$. This leads to $\mu_j^{c_i\bb_j}=\mathbf{t}^{\bb_j}$, where $\mu_j=({\mu}_{i1},\dots,{\mu}_{ik})$. Thus, we have $(\mu_j^{\bb_j}p_{ij})^{c_i}=1$, for all $j=1,\dots,r$. If $p'_{ij}=\mu_j^{\bb_j}p_{ij}$, then $p'_{ij}\in H$ and $[P'_i]=[P_i]$. In order to unify the rest of the proof for both cases that $\K$ is finite and algebraically closed, by abusing the notation, let us assume that $p_{ij}\in H$. This implies that $H$ is a non-empty finite subgroup of $\K$. Since $\K$ is a field and $H$ is finite, $H$ must be cyclic, i.e., $H=\la \eta \ra$, for some $\eta\in \K^*$. Then $p_{ij}=\eta^{q_{ij}}$ for some $q_{ij}$. Let $Q$ be the $s\times r$ integer matrix with entries $q_{ij}$ and let $$Y_{Q,H}=\{[t_1^{q_{11}}\cdots t_s^{q_{s1}}:\cdots :t_1^{q_{1r}}\cdots t_s^{q_{sr}}]: t_1,\dots,t_s \in H\}.$$

Then $$\begin{array}{lcl}
[P]\in Y &\iff& [P]=[P_1]^{i_1}\cdots[P_s]^{i_s}=[p_{11}^{i_1}\cdots p_{s1}^{i_s}:\dots:p_{1r}^{i_1}\cdots p_{sr}^{i_s}]\\
&\iff&[P]=[\eta^{i_1q_{11}}\cdots \eta^{i_sq_{s1}}:\dots:\eta^{i_1q_{1r}}\cdots \eta^{i_sq_{sr}}]
\end{array}
$$
Since $t_j\in H \iff t_j=\mu^{i_j}$, for all $j=1,\dots,s$, it follows that  $[P]\in Y \iff [P]\in Y_{Q,H}$. Therefore, $(2)$ holds.

Suppose now that $(2)$ holds. Then, taking $t_i=1\in H$, for all $i=1,\dots,s$, we see that $[1]\in Y$. The product of two elements $[P]=[t_1^{q_{11}}\cdots t_s^{q_{s1}}:\cdots :t_1^{q_{1r}}\cdots t_s^{q_{sr}}]$ and $[P']= [{t'}_1^{q_{11}}\cdots {t'}_s^{q_{s1}}:\cdots :{t'}_1^{q_{1r}}\cdots {t'}_s^{q_{sr}}]$ of $Y$ lies in $Y$ since $t_it'_i\in H$ and 
$$[P][P']=[(t_1 {t'}_1)^{q_{11}}\cdots (t_s {t'}_s)^{q_{s1}}:\cdots :(t_1 {t'}_1)^{q_{1r}}\cdots (t_s {t'}_s)^{q_{sr}}].$$ Therefore, $(1)$ holds, as $Y$ is finite whenever $H$ is so.
\end{proof}

 \begin{lm}\label{l:varietyoflattice}  If $L$ is a homogeneous lattice with basis $\mathbf b_1,\dots,\mathbf b_{\ell}$ then $$V_X(I_L)\cap T_X=\bigcap_{i=1}^{\ell} V_X(F_{\be_i})\cap T_X.$$
\end{lm}
\begin{proof} Take a point $[P]$ from the intersection $\bigcap_{i=1}^{\ell} V_X(F_{\be_i})\cap T_X$. Then, $F_{\be_i}(P)=\x^{{\be_i}^+}(P)-\x^{{\be_i}^-}(P)=0$, for all $i=1,\dots,\ell$. Since $[P]\in T_X$, its coordinates are not zero, and thus $\x^{\be_i}(P)=1$, for all $i=1,\dots,\ell$. As we have $\m=c_1\be_1+\cdots+c_{\ell}\be_{\ell}$, for any $\m\in L$, it follows that $\x^{\m}(P)=1$. So, $F_\m(P)=0$, for any $\m\in L$. Hence, $[P]\in V_X(I_L)\cap T_X$. The other inclusion is straightforward.
\end{proof}

For finite fields Theorem \ref{t:parametrizingsubmonoids} has a more precise form given below.
 \begin{pro}\label{p:parameterisingV_L}  Let $\K=\F_q$ be a finite field. Let $L\subset \Z^r$ be a homogeneous lattice of rank $\ell$ with basis vectors $\mathbf b_1,\dots,\mathbf b_{\ell}$ and $B_L$ be a matrix  with rows of the form $[\mathbf b_j \; \mathbf (q-1) \mathbf e_j]$ where $\mathbf e_j$ is the standard basis vector of $\Z^{\ell}$. If $A_L$ is an integer matrix of size $(r+\ell)\times r$ such that $\Ker (B_L)=\im (A_L)$, and $A=[a_{ij}]$ is an $r \times r$ matrix consisting of the first $r$ rows of $A_L$, then $V_X(I_L)\cap T_X=Y_{A,\K^*}$.
\end{pro}
\begin{proof} Since $\K$ is a finite field, $\K^*=\la \eta \ra$ is cyclic, and so every point in $T_X$ is of the form $[P]=[\eta^{s_1}:\cdots:\eta^{s_r}]$ for some $\mathbf s=(s_1,\dots,s_r)\in \N^r$. Then, by Lemma \ref{l:varietyoflattice}, we have $[P]\in V_X(I_L)$ if and only if $F_{\be_j}(P)=0$ for all $j=1,\dots,\ell$. This is equivalent to $\eta^{\mathbf s\cdot \be_j}=1$ for all $j=1,\dots,\ell$. Since the order of $\eta $ is $q-1$ the latter is equivalent to $\mathbf s\cdot \mathbf b_j \cong 0 \mod q-1$ for all $j=1,\dots,\ell$. This can be turned into a system of linear equations as $\mathbf s\cdot \mathbf b_j \cong 0 \mod q-1$ if and only if $\mathbf s\cdot \mathbf b_j =(q-1) c_j$ for some $c_j \in \Z$:
$$\begin{array}{lcl}
 s_1b_{11}+\cdots+s_rb_{1r}-(q-1)c_1&=&0\\
 &\vdots&\\
  s_1b_{\ell 1}+\cdots+s_rb_{\ell r}-(q-1)c_{\ell}&=&0.
\end{array}
$$
Then $B_L$ is the matrix of this system. Thus, $\mathbf s\in \Z^r$ is a common solution of the equations $\mathbf s\cdot \mathbf b_j \cong 0 \mod q-1$ for all $j=1,\dots,\ell$ if and only if $(\mathbf s,\mathbf c)\in \Ker (B_L)\cap \Z^{r+\ell}$ for some $\mathbf c=(c_1,\dots,c_{\ell}) \in \Z^{\ell}$. If $A_L$ is an integer matrix of size $(r+\ell)\times r$ such that $\Ker (B_L)=\im (A_L)$, then the last condition means that there is some $(\lambda_1,\dots,\lambda_{r})\in \Z^{r}$ with $(\mathbf s,\mathbf c)=A_L(\lambda_1,\dots,\lambda_r)$. If $A=[a_{ij}]$ is an $r \times r$ matrix consisting of the first $r$ rows of $A_L$, we have $[P]\in V_X(I_L)$ if and only if  $(s_1,\dots,s_{r})=A(\lambda_1,\dots,\lambda_{r})$. In this case, $s_i=a_{i1}\lambda_1+\cdots+a_{ir}\lambda_{r}$ and hence $\eta^{s_i}=(\eta^{\lambda_1})^{a_{i1}}\cdots (\eta^{\lambda_r})^{a_{ir}}$, for all $i=1,\dots,r$. Letting $t_j=\eta^{\lambda_j}$ for any $j=1,\dots,r$, we conclude that $[P]\in V_X(I_L)$ if and only if $[P]=[t_1^{a_{11}}\cdots t_r^{a_{1r}}:\cdots :t_1^{a_{r1}}\cdots t_r^{a_{rr}}]$. So, if $[P]\in V_X(I_L)\cap T_X$ then $[P]\in Y_{A,\K^*}$ by above. Converse is straightforward since any $t_j\in \K^*$ is of the form $\eta^{\lambda_j}$ for $j=1,\dots,r$.
\end{proof}

\begin{algorithm} \label{a:algorithm}
\caption{ Parameterizing the zero set of a lattice ideal in the torus.}
\begin{flushleft}
\hspace*{\algorithmicindent} \textbf{Input} A matrix L whose rows constitute a basis for a lattice $L\subset \Z^r$,\\
\hspace*{\algorithmicindent} \textbf{Output} A matrix $A$  parameterizing $V_X(I_L)\cap T_X$.
\end{flushleft}
\begin{algorithmic}[1]
 			\State Form the matrix $B_L$ adding to L an identity matrix of size rank L.
 			\State Find a matrix $A_L$ with $\Ker (B_L)=\im (A_L)$
 		    \State Choose the first $r$ rows of $A_L$
\end{algorithmic}
\end{algorithm}

The following will illustrate how to run this algorithm in Macaulay$2$. 
 \begin{ex}
	\label{ex:H2param} Let $X=\cl H_{2}$ be the Hirzebruch surface whose rays are generated by $\vv_1=(1,0)$, $\vv_2=(0,1)$, $\vv_3=(-1,2)$,
	and $\vv_4=(0,-1)$. Thus we have the short exact sequence 
	$$\dis \xymatrix{ \mathfrak{P}: 0  \ar[r] & \Z^2 \ar[r]^{\phi} & \Z^4 \ar[r]^{\bb}& \Z^2 \ar[r]& 0},$$  
	where $$\phi=\begin{bmatrix}
	1 & 0 & -1& ~~0 \\
	0 & 1 & ~~~2& -1
	\end{bmatrix}^T   \quad  \mbox{ and} \quad \bb=\begin{bmatrix}
	1 & -2 & 1& 0 \\
	0 & ~~~1 & 0&1  
	\end{bmatrix}.$$ 
	This shows that the class group is $\cA=\Z^2$ and the total coordinate ring is $S=K[x_1,x_2,x_3,x_4]$ where  $$\deg_{\cA}(x_1)=\deg_{\cA}(x_3)=(1,0), \deg_{\cA}(x_2)=(-2,1), \deg_{\cA}(x_4)=(0,1).$$ 
	
We illustrate Proposition \ref{p:parameterisingV_L} taking $L\subset \Z^4$ to be the lattice generated by the rows of the matrix L below.
\begin{verbatim}
q=11;
beta = matrix {{1, -2, 1, 0}, {0, 1, 0, 1}};
L=matrix {{10,0,-10,0},{0,5,10,-5}};
beta*transpose L--L is homogeneous as this is zero
r=numColumns L;
l=numRows L;
BL=L|(q-1)*id_(ZZ^l)
O = matrix mutableMatrix(ZZ,r,l)--zero matrix
A=transpose ((id_(ZZ^r)| O)*gens ker BL)

o1 = |  0  1  0  1 |
     |  1  0  0  0 |
     |  0  2 -1  0 |
     | -1  0 -1  0 | 
\end{verbatim}	
In this example the zero set $V_X(I_L)$ of the lattice ideal $I_L=\la x_1^{10}-x_3^{10}, x_2^{5}x_3^{10}-x_4^{5} \ra$ does lie in the torus, as $V(I_L)\setminus (\F_q^*)^4 \subset V(B)=V(x_1,x_3)\cup V(x_2,x_4)$. Therefore, $V_X(I_L)=Y_{A,\F_{11}^*}$.
\end{ex}
We also consider the following typical singular example.
\begin{ex} \label{ex:1113param} Let $X=P(1,1,1,3)$ be the weighted projective surface with homogeneous coordinate ring $S=\K [x_1,x_2,x_3,x_4]$ where $\deg_{\cA}(x_1)=\deg_{\cA}(x_2)=\deg_{\cA}(x_3)=1$, $\deg_{\cA}(x_4)=3$.
The same lattice is homogeneous here as in the previous example. 
\begin{verbatim}
q=11;
beta = matrix {{1,1,1,3}};
Phi=gens ker beta

o1 =  |  1  -1   0  |
      | -1   0   3  |
      |  0   1   0  |
      |  0   0  -1  |
L=matrix {{10,0,-10,0},{0,5,10,-5}};
beta*transpose L--L is homogeneous as this is zero
r=numColumns L;
l=numRows L;
BL=L|(q-1)*id_(ZZ^l)
O = matrix mutableMatrix(ZZ,r,l)--zero matrix
A=transpose ((id_(ZZ^r)| O)*gens ker BL)

o1 = |  0  1  0  1 |
     |  1  0  0  0 |
     |  0  2 -1  0 |
     | -1  0 -1  0 |
\end{verbatim}
Thus we have the short exact sequence 
	$$\dis \xymatrix{ \mathfrak{P}: 0  \ar[r] & \Z^3 \ar[r]^{\phi} & \Z^4 \ar[r]^{\bb}& \Z \ar[r]& 0},$$  
	where $$\phi=\begin{bmatrix}
	~~1 & -1 & ~~0& ~~0 \\
	-1 & ~~0 & ~~1& ~~0 \\
	~~0 & ~~3 & ~~0& -1
	\end{bmatrix}^T   \quad  \mbox{ and} \quad \bb=\begin{bmatrix}
	1 & 1 & 1& 3 \\  
	\end{bmatrix}.$$ 
In this example the zero set $V_X(I_L)$ of the lattice ideal does not lie in the torus. But if we add to $L$ the subgroup generated by the vector $(0,10,-10,0)$ and denote by $L'$ the new lattice we obtain, then it follows that $V_X(I_{L'})=V_X(I_L)\cap T_X=Y_{A,\F_{11}^*}$.
\end{ex}

We close this section with the following observation.

\begin{coro} If $\K=\F_q$, then $Y_{Q,H}=Y_{A,\K^*}$ for some square integer matrix $A$.
\end{coro}
\begin{proof}
Let $I_L=I(Y_{Q,H})$ be the lattice ideal of $Y_{Q,H}\subseteq T_X$. Then, by Proposition \ref{p:parameterisingV_L}, there is a square integer matrix $A$ such that $V_X(I_L)\cap T_X=Y_{A,\K^*}$. Since $Y_{Q,H}=V_X(I_L)$ by Lemma \ref{L:Zariski}, the claim follows.
\end{proof}

\section{Degenerate Tori}		
In this section, we focus on finite submonoids of the torus $T_X$ which are parameterized by diagonal matrices.
\begin{defi} If $H=\la \eta \ra$ is a finite submonoid of $\K^*$, then the finite submonoid $Y_{A,H} =\{[t_1^{a_1}:\cdots:t_r^{a_r}] \: : \: t_i \in H \}$ of the torus $T_X$ is called a {\bf degenerate torus}. 
\end{defi}
Since $H=\la \eta \ra$, every $t_i \in H$ is of the form $t_i=\eta^{s_i}$, for some $0 \leq s_i \leq h-1$, where $h=| \eta |$ is the order of $\eta$. Let $d_i=| \eta^{a_i} |$ and $D=diag(d_1,\dots,d_r)$. 

\begin{pro}\label{p:degY}  If $d_1,\dots,d_r$ are coprime to each other then $Y_{A,H}$ is a cyclic group of order $d_1 \cdots d_r$. 
\end{pro} 

\begin{proof} Let $[P_i]=[1:\cdots:\eta^{a_i}:\cdots:1]$, for $i=1,\dots,r$. If $[P_i]^{s_i}=[1]$, then $(1,\dots,1,\eta^{a_is_i},1,\dots,1)\in G$. Since $G=\Ker \pi$, it follows that $\x^{\uu_1}=\cdots=\x^{\uu_n}=1$ at $[P_i]$. So, $\eta^{a_is_iu_{1i}}=\cdots=\eta^{a_is_iu_{ni}}=1$. Since $\vv_i=(u_{1i},\dots,u_{ni})$ is a primitive ray generator, the greatest common divisor of its coordinates is $1$. Thus, $d_i$ divides $s_i$ as $d_i=| \eta^{a_i} |$. This means the point $[P_i]$ generates a cyclic group of order $d_i$, for all $i=1,\dots,r$. 

On the other hand, every element $[t_1^{a_1}:\cdots:t_r^{a_r}]$ is a product $[P_1]^{s_1}\cdots [P_r]^{s_r}$ of powers of the points $[P_1], \dots [P_r]$, as $t_i=\eta^{s_i}$, for all $i=1,\dots,r$. This shows that $Y$ is a product of cyclic subgroups generated by $[P_1],\dots, [P_r]$. Under the hypothesis, this product becomes a direct product, completing the proof.
\end{proof}

\begin{rema} When $X=\mathbb{P}^{r-1}$, $H=\F_q^*$, and $\gcd(d_1,\dots,d_r)=1$, the order of $Y_{A,H}$ is $d_1 \cdots d_r$ by \cite{LVZ}. This does not hold in general, as the following example illustrates. 
\end{rema}
\begin{ex} \label{ex:notcyclic}
Take $X=\cl H_{2}$ to be the Hirzebruch surface over $\F_{11}$ and consider $(a_1,a_2,a_3,a_4)=(2,5,4,5)$. Then, we have $(d_1,d_2,d_3,d_4)=(5,2,5,2)$ and $\gcd(d_1,d_2,d_3,d_4)=1$. In this case $G$ has points of the form $(g_1,g_2,g_1,g_1^2g_2)$ for $g_1,g_2 \in \K^*$. Denote by $C_i$ the cyclic subgroup generated by $[P_i]$, for $i=1,2,3,4$. One can check that $C_1\cap C_3=[1]$, so $C_1C_3=C_1 \times C_3$. Since the orders $25$ of $C_1C_3$ and $2$ of $C_2$ are relatively prime, it follows that $C_1C_2C_3=C_1 \times C_2\times C_3$ having order $50$. Now, $C_4$ is a subgroup of $C_1C_2C_3$, since $[P_4]=[P_2]$ via $(1,1,1,\eta^{5})=(1,\eta^{5},1,\eta^{5})(1,\eta^{5},1,1)$. Hence, $Y_{A,H}=C_1C_2C_3C_4=C_1C_2C_3$ have order $50$. 

On the other hand, if we take $(a_1,a_2,a_3,a_4)=(5,2,5,4)$, then 
$(d_1,d_2,d_3,d_4)=(2,5,2,5)$ and $\gcd(d_1,d_2,d_3,d_4)=1$ as before. We observe that $[P_3]=[P_1]$ and $C_4=C_2$ in this case. So, 
$Y_{A,H}=C_1C_2C_3C_4=C_1\times C_2$ have order $10$. 
\end{ex}

As $Y_{A,H}$ is a submonoid of $T_X$, $I(Y_{A,H})$ is a lattice ideal. We are ready to give our first main result specifying the lattice of this ideal. Algorithms for finding general $I(Y)$ are given in \cite{BaranSahin}.
\begin{tm}\label{t:I(Y_A)}  If $Y=Y_{A,H}$ then $I(Y)=I_L$ for  $L=D(L_{\beta D})$. 
\end{tm} 

\begin{proof} Take a generator $F_\m=\x^{\m^{+}}-\x^{\m^{-}} $ of $I(Y_{A,H})$. Since $F_\m$ homogeneous, $\m \in L_{\bb}$ and $F_{\m}(P)=0$, for all $[P]\in Y_{A,H}$. So, $\x^{\m}(P)=1$, for all $[P]\in Y_{A,H}$. This is true in particular for points having homogeneous coordinates $[1:\cdots:\eta^{a_i}:\cdots:1]$ with $\eta^{a_i}$ in the $i-th$ component.  So, $(\eta^{a_i})^{m_i}=1$ yielding $d_i | m_i$, i.e., $\m=D(\m')$, for some $\m'\in \Z^r$. As $\m \in L_{\bb}$, we have $\bb D(\m')=\bb \m=0$, so $\m'\in L_{\bb D}$. Therefore, $\m \in D(L_{\beta D})=L$ and $F_\m$ is a binomial generator of $I_{L}$.

Conversely, take a binomial generator $F_\m$ of $I_{L}$. Then $\m \in L$, so $\m=D(\m')$, for some $\m'\in L_{\bb D}$. Since $\bb \m=\bb D(\m')=0$,  $F_\m$ is homogeneous. On the other hand, $\x^{\m}([P])=\eta^{a_1s_1m_1}\cdots \eta^{a_rs_rm_r}=1$ , for all $[P]=[\eta^{a_1s_1}:\cdots:\eta^{a_rs_r}] \in Y_{A,H}$, since  $m_i=d_im'_i$ and $(\eta^{a_i})^{d_i}=1$ for all $i \in [r]$. Thus, $\x^{\m^{+}}([P])=\x^{\m^{-}} ([P])$ and $F_\m([P])=0$, for all $[P] \in Y_{A,H}$. Hence, $F_\m$ is a generator of $I(Y_{A,H})$.
\end{proof}

\begin{rema} As $d_i=| \eta^{a_i} |=h/ \gcd(h,a_i)$, the lattice $L=D(L_{\beta D})$ depends on $H$. Since the lattice $L_{\beta D}$ is saturated, $I_{L_{\beta D}}$ is a prime lattice ideal also known as a toric ideal. There are various algorithms for finding generators of a toric ideal, and these can be used here.
\end{rema}

	\begin{algorithm} \label{a:algorithm2}
 		\caption{ Lattice of the vanishing ideal of a degenerate torus.}
\begin{flushleft}
		\hspace*{\algorithmicindent} \textbf{Input} A matrix A parameterizing the degenerate torus,\\
 		\hspace*{\algorithmicindent} \textbf{Output} A matrix $\Gamma$ whose columns generate the lattice $L$.
	\end{flushleft} 
 		\begin{algorithmic}[1]
 			\State Find the matrix $D$ using $d_i=h/ \gcd(h,a_i)$.
 			\State Compute the matrix LbetaD whose columns generate $L_{\beta D}$.
 		    \State Multiply D by LbetaD 
 		\end{algorithmic}
 	\end{algorithm}

\begin{defi} If each column of a matrix has both a positive and a negative entry we say that it is {\it mixed}. If it does not have a square mixed submatrix, then it is called {\it dominating}. 
\end{defi}

\begin{ex} The matrix $\phi=\begin{bmatrix}
 0 & 1 &1& ~~1& ~~0 & -1& -1& -1 \\
 1 & 1 &0& -1 & -1& -1 & ~~0& ~~1
\end{bmatrix}^T$ is a mixed BUT not dominating matrix.
\end{ex}

\begin{defi}\label{D:CI} 
For $Y\subset X$, the ideal $I(Y)$ is called a {\it complete intersection} if it is generated by
a regular sequence of homogeneous polynomials $F_1,\dots, F_k\in S$ 
where $k$ is the height of $I(Y)$.
\end{defi}

\begin{tm} [Morales-Thoma 2005]\label{t:MT} Let $L$ be a sublattice of $\Z^r$ with $L\cap \N^r=\{0\}$ and $\Gamma$ be a matrix whose columns constitute a basis of $L$. Then $I_L$ is a complete intersection iff $\Gamma$ is mixed dominating.
\end{tm}

\begin{ex} \label{ex:idealnotcyclic}
Take $X=\cl H_{2}$ to be the Hirzebruch surface over $\F_{11}$ and consider $(a_1,a_2,a_3,a_4)=(2,5,4,5)$. Then, we have $(d_1,d_2,d_3,d_4)=(5,2,5,2)$ and $\gcd(d_1,d_2,d_3,d_4)=1$. We compute the vanishing ideal in Macaulay2 as follows:
\begin{verbatim}
i1 :  beta = matrix {{1, -2, 1, 0}, {0, 1, 0, 1}}

o1 = | 1 -2 1 0 |
     | 0 1  0 1 |
     
i2 : D=matrix{{5,0,0,0},{0,2,0,0},{0,0,5,0},{0,0,0,2}};

i3 : LbetaD =gens ker (beta*D);

i4 : gamma=D*LbetaD 

o4 =  | -5   20  |
      |  0   10  |
      |  5    0  |
      |  0  -10  |
\end{verbatim}

Thus, $I(Y_{A,\K^*})=\la x_1^5-x_3^5, x_1^{20}x_2^{10}-x_4^{10} \ra$. 

On the other hand, if we take $(a_1,a_2,a_3,a_4)=(5,2,5,4)$, then 
$(d_1,d_2,d_3,d_4)=(2,5,2,5)$ and $\gcd(d_1,d_2,d_3,d_4)=1$ as before. Similarly, the vanishing ideal is found to be $I(Y_{A,\K^*})=\la x_1^2-x_3^2, x_1^{10}x_2^{5}-x_4^{5} \ra$. Note that the matrix Gamma above is mixed-dominating and that the ideal $I(Y_{A,\K^*})$ is a complete intersection, in both cases.
\end{ex}

\begin{pro} \label{p:ci} A generating system of binomials for $I(Y_{A,H})$ is obtained from that of $I_{L_{\beta D}}$ by replacing $x_i$ with $x_i^{d_i}$. $I(Y_{A,H})$ is a complete intersection if and only if so is the toric ideal $I_{L_{\beta D}}$. In this case, a minimal generating system is obtained from a minimal generating system of $I_{L_{\beta D}}$ this way.
\end{pro} 

\begin{proof} $I(Y_{A,H})=I_L$ for $L=D(L_{\beta D})$. The first part relies on the fact that $\m\in L_{\beta D}$ iff $D\m\in D(L_{\beta D})$ and thus $\x^{\m^{+}}-\x^{\m^{-}} \in I_{L_{\beta D}}$ iff $\x^{D\m^{+}}-\x^{D\m^{-}} \in I(Y_{A,H})$. So, there is a one-to-one correspondence between the binomials with disjoint supports. If $\x^{\m^{+}}-\x^{\m^{-}}=g_1(\x^{\m_1^{+}}-\x^{\m_1^{-}})+\cdots+g_k(\x^{\m_k^{+}}-\x^{\m_k^{-}})$, then $\x^{D\m^{+}}-\x^{D\m^{-}}=D(g_1)(\x^{D\m_1^{+}}-\x^{D\m_1^{-}})+\cdots+D(g_k)(\x^{D\m_k^{+}}-\x^{D\m_k^{-}})$, where $D(g)$ means that every $x_i$ in the polynomial $g$ is replaced with $x_i^{d_i}$.

Let $\Gamma$ be a matrix whose columns constitute a basis for the lattice $L_{\beta D}$. It is clear that $D\Gamma$ is a matrix whose columns form a basis of $L$. Since $D$ is a diagonal matrix with positive entries in the main diagonal, the sign patterns of $\Gamma$ and $D\Gamma$ are the same. By Theorem \ref{t:MT} above, $I(Y_{A,H})$ is a complete intersection iff
so is the toric ideal $I_{L_{\beta D}}$. In this case we know that the minimal generating sets for both ideals have the same cardinality : height $I(Y_{A,H})=$ height $I_{L_{\beta D}}$. If $f_1,\dots,f_k$ are minimal generators for the ideal $I_{L_{\beta D}}$, with $k=$ height $I_{L_{\beta D}}$, then $D(f_1),\dots,D(f_k)$ are generators for the ideal $I(Y_{A,H})$. Since height $I(Y_{A,H})$ is a lower bound for the number of generators, this set must be a minimal generating set, completing the proof.
\end{proof}

\begin{rema} The toric ideal $I_{L_{\beta D}}$ depends on $h$ but not on the field $\K$.
\end{rema} 
When $\cA$ is free, the ideal of $\{[1]\}$ is the toric ideal of $\N\beta$.
\begin{coro} \label{c:torus}When $\K=\F_q$ we have the following.\\
(\textbf{i}) If $Y=\{[1]\}$ then $I(Y)=I_{L_{\beta}}$,\\
(\textbf{ii}) If $Y=T_X$ then $I(Y)=I_{L}$, for $L=(q-1)L_{\beta}$,\\ 
(\textbf{iii}) $I(T_X)$ is a complete intersection iff so is $I_{L_{\beta}}$, which is independent of $q$.
\end{coro} 

\begin{proof} (\textbf{i}) If $Y=\{[1]\}$ then $H=\la 1 \ra $. So, $D=I_r$ and $I(Y)=I_{L_{\beta}}$. This is true for any field $\K$.\\
(\textbf{ii}) If $Y=T_X$ then $H=\K^*$ and $A=I_r$. So, $D=(q-1)I_r$ and thus $I(Y)=I_{L}$ for $L=DL_{\beta D}=(q-1)L_{\beta}$, since $\beta D=(q-1)I_r\beta=(q-1)\beta$ and $\beta$ have the same kernel: $L_{\beta D}=L_{\beta}$.\\ 
(\textbf{iii}) This is a direct consequence of Proposition \ref{p:ci}.
\end{proof}

These consequences generalize some results of \cite{DN} from weighted projective spaces to a general toric variety. Using the matrix $\phi$ defined by the fan $\Sig$ and the result presented in this section one can easily check whether the vanishing ideal of $T_X$ is a complete intersection.


\begin{ex} The matrix $\phi=\begin{bmatrix}
 1 & 0 & -1& ~~0 \\
0 & 1 & ~~~\ell& -1
\end{bmatrix}^T$ is a mixed dominating matrix.

\begin{itemize}

\item So, $I_{L_\bb}=\la {x_1}-{x_3},{x_2}{x_3}^{\ell}-{x_4}^{}\ra$ is a complete intersection.

\item Thus, tori $T_X$ of the Hirzebruch surfaces $\cl H_{\ell}$ are all complete intersection for every $q$ and $\ell$: 
$$I(T_X)=\la {x_1}^{q-1}-{x_3}^{q-1},{x_2}^{q-1}{x_3}^{\ell(q-1)}-{x_4}^{q-1}\ra.$$


  
 \end{itemize}
\end{ex}

\section{Primary Decomposition of Vanishing Ideals}
In this section, we list main properties of vanishing ideals of subsets of the torus $T_X$ over $\F_q$.
\begin{tm} \label{t:primdec} Let $Y$ be a subset of $T_X$ and $[P]\in T_X$.\\
(\textbf{i}) If $I(Y)=I_L$ then $I([P]Y)=\la \x^{\m^+}-\x^{\m}(P)\x^{\m^-}  ~ | ~  \m \in L \ra$.\\
(\textbf{ii}) $I([P])=\la \x^{\m^+}-\x^{\m}(P)\x^{\m^-}  ~ | ~ \m \in L_{\beta} \ra$ is a prime ideal of height $n$.\\ 
(\textbf{iii}) The minimal primary decomposition of $I(Y)$ is as follows:
$$\displaystyle I(Y)=\bigcap_{[P] \in Y} I([P]).$$ 
(\textbf{iv}) $I(Y)$ is radical.\\ 
(\textbf{v}) $\height I(Y)=n$.\\
(\textbf{vi}) $\dim I(Y)=r-n$. 
\end{tm}

\begin{proof}
(\textbf{i}) Let $F_p(\x):=F(p^{-1}\x)$, for a polynomial $F\in S$. Then $F_{p}\in I([P]Y)$ if and only if $F\in I(Y)$, since $F_p(p\t)=F(p^{-1}p\t)=F(\t)$ for all $\t\in Y$. If $I(Y)=I_L$ and $F(\x)=\x^{\m^+}-\x^{\m^-}$ is a generator for $I_L$, then $$F_p(\x)=\x^{-\m^+}(P)\x^{\m^+}-\x^{-\m^-}(P)\x^{\m^-}=\x^{-\m^+}(P)(\x^{\m^+}-\x^{\m}(P)\x^{\m^-}),$$
giving a generator $\x^{\m^+}-\x^{\m}(P)\x^{\m^-}$ for $I([P]Y)$. When $[P]=[P']$, $p=gp'$, for some $g\in G$. Since $\x^{\m}(P)=\x^{\m}(g)\x^{\m}(P')=\x^{\m}(P')$, for all $g\in G$ and $\m\in L \subset L_\bb$, it follows that $\x^{\m^+}-\x^{\m}(P)\x^{\m^-}$ is independent of the representative of $[P]$.

(\textbf{ii}) Since $\{[P]\}=[P]Y$ for $Y=\{[1]\}$, the claim follows from the previous item and Corollary \ref{c:torus} (\textbf{i}). As the map sending $F$ to $F_p$ is ring autormorphism, it follows that $I([P])$ is a prime ideal of height $n$ if and only if so is the toric ideal $I([1])=I_{L_\bb}$. By \cite[Proposition 7.5]{MS}, the Krull dimension of $S/I_{L_\bb}$, which is $r-\height I_{L_\bb}$, equals $r-\rank {L_\bb}=r-n$, whence the result. 

(\textbf{iii}) Since intersection of homogeneous ideals is homogeneous, we have $$\displaystyle I(Y)=\bigcap_{[P] \in Y} I([P]).$$	
In order to show that this decomposition is minimal we prove that $I([P])$ is a minimal prime of $I(Y)$, for every $[P]\in Y$. This follows from the following observation. $ (\x^{\m^+}-\x^{\m}(P)\x^{\m^-})(P_0)=0$ if and only if $\x^{\m}(P)=\x^{\m}(P_0)$ if and only if $ (\x^{\m^+}-\x^{\m}(P_0)\x^{\m^-})(P)=0$. Thus, $I([P])\subset I([P_0])$ if and only if $I([P_0])\subset I([P])$, meaning that $I([P])\subset I([P_0])$ if and only if $I([P])= I([P_0])$.

(\textbf{iv}) This follows directly from the previous decomposition and the fact that each $I([P])$ is prime.
  
(\textbf{v}) $\height I(Y)=n$, since all the minimal primes are of height $n$ and thus the maximum is $n$.

(\textbf{vi}) $\dim I(Y)$ is the maximum of $\dim I([P])$ as $[P]$ vary in $Y$. Since $\dim I([P])=r-\height I([P])$, for all points $[P] \in Y$, the result follows from item (\textbf{ii}).
\end{proof}

\section{Evaluation codes on degenerate tori} \label{S:codes} 
In this section we apply our results in previous sections to
evaluation codes on degenerate tori in a toric variety. 

Recall the basic definitions from coding theory. Let $\F_q$ be a finite field of $q$ elements and 
$\F_q^*=\F_q\setminus\{0\}$ its multiplicative group. 
A subspace $\cC$ of $ \F_q^{N}$ is called a \textit{linear code}, and its elements 
${c}=(c_{1},\dotsc,c_{N})$ are called \textit{codewords}.  The number $N$ is called the {\it block-length} of $\cC$.
The {\it weight} of $c$ in $\cC$ is the number of non-zero entries in $c$.
The {\it distance} between two codewords $a$ and $b$ in $\mathcal{C}$ is the weight of $a-b\in\cC$.
The minimum distance between distinct codewords in $\cC$ is the same as the minimum weight of  non-zero codewords in $\cC$. The block-length $N$, the dimension $k=\dim_{\F_q}(\cC)$, and the minimum
distance $d=d(\cC)$ are the basic parameters of $\cC$.

Now, let $X$ be a simplicial complete toric variety over $\F_q$ with torsion-free class group and $S$ its 
homogeneous coordinate ring as in the previous sections. 

  Now, we recall evaluation codes defined on subsets $Y=\{[P_1],\dots,[P_N]\}$ of the torus $T_X$. Fix a degree $\aa\in\N\beta$ and a monomial $F_0=\x^{\phi(\m_0)+\a} \in S_{\aa}$,
where  $\m_0\in \Z^n$, $\a$ is any element of $\Z^r$ with $\deg(\a)=\aa$, and $\phi$ as in the exact sequence $\mathfrak{P}$.
This defines the {\it evaluation map}
\begin{equation}\label{e:evalmap}
\text{ev}_{Y}:S_\aa\to \F_q^N,\quad F\mapsto \left(\frac{F(P_1)}{F_0(P_1)},\dots,\frac{F(P_N)}{F_0(P_N)}\right).
\end{equation}
The image $\cC_{\aa,Y}=\text{ev}_{Y}(S_\aa)$ is a linear code, called the {\it generalized toric code}.
It can be readily seen that different choices of $F_0\in S_\aa$ yield to equivalent codes. Clearly, the block-length $N$ of $\cC_{\aa,Y}$ equals $|T_X|=(q-1)^n$ in the case of toric codes where $Y=T_X$, which was introduced for the first time by Hansen in \cite{Ha0, Ha1}. A way to compute the dimension of a toric code is given in \cite{Ru}. In the generalized case, multigraded Hilbert function can be used to compute the length and dimension as shown by \c{S}ahin and Soprunov in \cite{sasop}. For general information about algebraic geometry codes we refer the reader to \cite{TVN} and 
\cite{Li}.


 The next proposition provides a way to calculate the dimension of the code as the
 value of the Hilbert function $H_Y(\aa):=\dim_{\K} S_{\aa}-\dim_{\K} I(Y)_{\aa}$. 
 
\begin{pro}\label{p:dim=HF} 
$\dim_{\K} \cC_{\aa,Y}$ equals $H_{Y}(\aa)$.
\end{pro}
\begin{proof} This follows as the kernel of the evaluation map is exactly the the degree $\aa$ part of the vanishing ideal $I(Y)$. 
\end{proof}

We write $\aa \preceq \aa'$ if $\aa'-\aa$ lies in $\N\beta$, where $\N\beta\subset \cA$ denotes the semigroup generated by the degrees $\bb_i$ of the variables $x_i$. 
\begin{coro}\label{c:dimcode} If the set $Y=Y_{A,H}$ and $\aa \preceq d_1\bb_1+\cdots +d_r \bb_r$, then $ev_{Y}$ is injective, that is, $\dim_{\K} \cC_{\aa,Y}=\dim_{\K} S_{\aa}$.
\end{coro}
\begin{proof} In this case we have $I(Y)=I_L$ for $L=D(L_{\bb D})$ by Theorem \ref{t:I(Y_A)}, where $D=diag (d_1,\dots,d_r)$. By Proposition \ref{p:ci}, the generators of $I(Y)$ are obtained from the generators of the toric ideal $I_{L_{\bb}}$ by substituting $x_i^{d_i}$ in $x_i$. Thus, their degrees are at least $d_i\bb_i$. In particular, the degree $\aa$ part of the kernel $I(Y)_{\aa}=\{0\}$ when $\aa \preceq d_1\bb_1+\cdots +d_r \bb_r$.
\end{proof}

\begin{coro}\label{c:lengthcode} If the set $Y=Y_{A,H}$ and $d_1,\dots ,d_r$ are pairwise relatively prime, then the length of $\cC_{\aa,Y}$ is $d_1\cdots d_r$.
\end{coro}
\begin{proof} This follows directly from Proposition \ref{p:degY}.
\end{proof}
We finish with the following example.

\begin{ex} \label{ex:idealnotcyclicCode}
Take $X=\cl H_{2}$ to be the Hirzebruch surface over $\F_{11}$ and consider $(a_1,a_2,a_3,a_4)=(2,5,4,5)$. Then, we have $(d_1,d_2,d_3,d_4)=(5,2,5,2)$ and $\gcd(d_1,d_2,d_3,d_4)=1$. Thus, $I(Y_{A,\K^*})=\la x_1^5-x_3^5, x_1^{20}x_2^{10}-x_4^{10} \ra$ by Example \ref{ex:idealnotcyclic}. The fan $\Sig$ of $X$ determines an important subsemigroup $\cl K$ of the semigroup $\N\bb$. Namely, $\dis \cl K=\cap_{\sig \in \Sig} \N\hat{\sig}$, where $\N\hat{\sig}$ is the semigroup generated by the subset $\{\bb_j \; : \; \rho_j \notin \sig\}$. In this example, $\dis \cl K=\N^2$. Since the degrees $(5,0)$ and $(0,10)$ lie in $\cl K$, it follows that the Hilbert function at $(5,10)$ gives us the length by \cite{sasop}.
\begin{verbatim}
i1 : hilbertFunction({5,10},IYQ)
o1 = 50
\end{verbatim}

On the other hand, if we take $(a_1,a_2,a_3,a_4)=(5,2,5,4)$, then 
$(d_1,d_2,d_3,d_4)=(2,5,2,5)$ and $\gcd(d_1,d_2,d_3,d_4)=1$ as before. Similarly, the vanishing ideal is found to be $I(Y_{A,\K^*})=\la x_1^2-x_3^2, x_1^{10}x_2^{5}-x_4^{5} \ra$. Note that the matrix Gamma above is mixed-dominating and that the ideal $I(Y_{A,\K^*})$ is a complete intersection, in both cases.

Dimensions of various codes are encoded in the matrix below:
\begin{verbatim}
i2 : apply(6,j-> apply(18,i-> hilbertFunction({i-5,5-j},IYQ)));
i3 : oo / print @@ print
{06,07,08,09,10,10,10,10,10,10,10,10,10,10,10,10,10,10}
{04,05,06,07,08,09,10,10,10,10,10,10,10,10,10,10,10,10}
{02,03,04,05,06,07,08,08,08,08,08,08,08,08,08,08,08,08}
{00,01,02,03,04,05,06,06,06,06,06,06,06,06,06,06,06,06}
{00,00,00,01,02,03,04,04,04,04,04,04,04,04,04,04,04,04}
{00,00,00,00,00,01,02,02,02,02,02,02,02,02,02,02,02,02}
\end{verbatim}
\end{ex}

\section*{Acknowledgements} All the examples were computed by using the computer algebra system Macaulay $2$, see \cite{Mac2}.
\bibliographystyle{amsplain}

\end{document}